\documentclass[a4paper, reqno, 11pt]{amsart}
\usepackage{amsmath, amsfonts, amssymb, amsthm, hyperref, color}
\usepackage{mathrsfs}
\newtheorem{thm}{Theorem}
\newtheorem{lemma}[thm]{Lemma}
\newtheorem{cor}[thm]{Corollary}
\newtheorem{rem}[thm]{Remark}
\newtheorem*{distinct_spheres_lemma}{The Distinct Spheres Lemma}
\newtheorem*{InfiniteMotionConjecture}{The Infinite Motion Conjecture}

\newcounter{temp}
\newenvironment{citethm}[1]
{
 \addtocounter{temp}{1}
 \newtheorem*{\alph{temp}theorem}{Theorem #1}
 \begin{\alph{temp}theorem}
}
{\end{\alph{temp}theorem}}
\newenvironment{citecor}[1]
{
 \addtocounter{temp}{1}
 \newtheorem*{\alph{temp}theorem}{Corollary #1}
 \begin{\alph{temp}theorem}
}
{\end{\alph{temp}theorem}}

\providecommand{\aut}{\mathop{\rm Aut \,}\nolimits}
\providecommand{\sym}{\mathop{\rm Sym \,}\nolimits}
\providecommand{\symdiff}{\mathop{\Delta}}
\newcommand{\N}{\mathbb{N}}

\providecommand{\cart}{\mathop{\Box}}

\renewcommand{\\}{\vspace{3mm}}

\title{\bf Distinguishability of infinite groups and graphs}

\author{Simon M. Smith}
\address{Simon M. Smith, Department of Mathematics, Syracuse University \\
   Syracuse, New York, U.S.A}
\email{simon.smith@chch.oxon.org, smsmit13@syr.edu}

\author{Thomas W. Tucker}
\address{Thomas W. Tucker, Department of Mathematics, Colgate University \\
   Hamilton, New York, U.S.A.}
\email{ttucker@colgate.edu}

\author{Mark E. Watkins}
\address{Mark E. Watkins, Department of Mathematics, Syracuse University \\
   Syracuse, New York, U.S.A}
\email{mewatkin@syr.edu}

\date{\today}
\begin{document}
\maketitle

\begin{abstract} 
The {\em distinguishing number} of a group $G$ acting faithfully on a set $V$ is the least number of colors needed to color the elements of $V$ so that no non-identity element of the group preserves the coloring. The {\em distinguishing number} of a graph is the distinguishing number of its full automorphism group acting on its vertex set. A connected graph $\Gamma$ is said to have {\em connectivity 1 } if there exists a vertex $\alpha \in V\Gamma$ such that $\Gamma \setminus \{\alpha\}$ is not connected. For $\alpha \in V$, an orbit of the point stabilizer $G_\alpha$ is called a {\em suborbit} of $G$.

We prove that every connected primitive graph with infinite diameter and countably many vertices has distinguishing number $2$. Consequently, any infinite, connected, primitive, locally finite graph is $2$-distinguishable; so, too, is any infinite primitive group with finite suborbits. We also show that all denumerable vertex-transitive  graphs of connectivity 1 and all Cartesian products of connected denumerable graphs of infinite diameter have distinguishing number $2$.
All of our results follow directly from a versatile lemma which we call The Distinct Spheres Lemma.
 
\end{abstract}

%
%
\section{Introduction}

The {\em distinguishing number} of a group of permutations $G$ of a set $V$ is the least number of colors needed to color the elements of $V$ so that no non-identity element of $G$ fixes every color class setwise. In particular, $G$ has distinguishing number $2$ if and only if there exists a subset $Y \subseteq V$ whose setwise stabilizer $G_{\{Y\}}$ is trivial.  The group $G \leq \sym(V)$ is said to be {\em $n$-distinguishable} if its distinguishing number is at most $n$.   A graph $\Gamma$ is said to be $n$-{\em distinguishable} if its full automorphism group $\aut(G)$ is $n$-distinguishable when acting on $V\Gamma$.

Recall that a group $G \leq \sym(V)$ is primitive if $G$ acts transitively on $V$ and the only $G$-invariant equivalence relations on $V$ are the trivial relation and the universal relation. A graph $\Gamma$ is {\em primitive} if its full automorphism group acts primitively on its vertex set. Primitive graphs with at least one edge are connected.  The graph $\Gamma$ is {\em denumerable} if its vertex set is denumerable (i.e., countably infinite). 

\\

The concept of distinguishing numbers was introduced in 1996 by Albertson and Collins \cite{AC1} in the context of finite graphs  and later generalized to general group actions. Since then, there have been myriad papers investigating the distinguishing number of specific classes of graphs. The greatest interest appears to be focused upon those with distinguishing number $2$, partly because this property requires the automorphism group to have a regular orbit on the power set of the vertex set (see \cite[Proposition 2.7]{BaCa}) and partly because this property tends to be generic for a variety of combinatorial and algebraic structures (see \cite{BaCa} and \cite{CT}, for example).

Some classes of infinite graphs known to have distinguishing number $2$ include all locally finite trees having no 1-valent vertex \cite{watkins:distinguishability}; the infinite hypercube of dimension $n$, the denumerable  random graph, and various ``tree-like" graphs (which are not necessarily locally finite) \cite{imrich_et_al:distinguishing}; and locally finite maps \cite{tucker:distinguishing_maps}.
With the present article, we add several more to the list, including the following.

\begin{thm} \label{thm:primitive_distinguishing_number_2} Every connected denumerable primitive graph with infinite diameter has distinguishing number $2$.
\end{thm}

\begin{cor} \label{primgraph_dist=2} Every connected infinite, primitive, locally finite graph has distinguishing number $2$.
\end{cor}

\begin {thm} \label{thm:vertex-trans} Every denumerable vertex-transitive graph of connectivity 1 has distinguishing number $2$.
\end{thm}

\begin{thm} \label{thm_cart_prod} The Cartesian product of any two connected denumerable graphs of infinite diameter has distinguishing number $2$.
\end{thm}

Our results can be applied to permutation groups.  A {\em suborbit} of a group $G \leq \sym(V)$ is an orbit of a point stabilizer $G_\alpha$ of $G$; that is, a set of the form $\beta^{G_\alpha} = \{\beta^g :  g \in G_\alpha\}$ where $\alpha, \beta \in V$.

If $G$ is a group of permutations of $V$ and $\alpha,\beta\in V$, then the components of the graph $(V, \{\alpha, \beta\}^G)$ with vertex set $V$ and edge set $\{\alpha, \beta\}^G = \{ \, \{\alpha^g, \beta^g\} :  g \in G\}$ are the equivalence classes of a $G$-invariant equivalence relation. Graphs of this form are called {\em orbital graphs} of $G$.  Thus if $G$ is primitive, then every such orbital graph $(V, \{\alpha, \beta\}^G)$ must be connected whenever  $\alpha\neq\beta$. It is easy to see that for a primitive group $G$, these graphs are all locally finite if and only if the suborbits of $G$ are all finite.

\\

Every finite primitive permutation group of degree $n>32$, other than the symmetric group or alternating group, has distinguishing number $2$ \cite[Theorem 1]{akos_seress}.  Furthermore, for primitive permutation groups of degree $n$ not containing the alternating group, the proportion of subsets with trivial setwise stabilizer tends to $1$ as $n\rightarrow \infty$ \cite{BaCa}. 

Thus the following corollaries of Theorem~\ref{thm:primitive_distinguishing_number_2} should not be surprising.

\begin{cor} \label{cor:distinguish_inf_diam_prim_groups} If $V$ is a denumerable set and $G \leq \sym(V)$ is primitive and has a connected
orbital graph of infinite diameter, then the distinguishing number of $G$ is
$2$.
\end{cor}

\begin{cor} \label{cor:distinguish_subdegree_finite_prim_groups} If $G\leq \sym(V)$ is infinite and primitive and all suborbits of $G$ are finite, then $G$ has distinguishing number $2$.
\end{cor}

All the results in this paper follow with little effort from a versatile lemma which we call the Distinct Spheres Lemma. It provides us with a very simple criterion,  the ``Distinct Spheres Condition," that is sufficient for $2$-distinguishability of connected denumerable graphs. 

A graph $\Gamma$ is said to satisfy {\bf The Distinct Spheres Condition}, or {\bf DSC}, if there exists a vertex $\alpha\in V\Gamma$ such that for all distinct $\gamma,\delta\in V\Gamma$,
\begin{equation*}
d(\alpha,\gamma)=d(\alpha,\delta) {\text{ implies }}  S(\gamma,n)\neq S(\delta,n)\ {\text{for infinitely many }}n\in\N,
\end{equation*}
where $S(\gamma, n)$ denotes the set of vertices of $\Gamma$ at distance $n$ from $\gamma$.
\vskip.2cm
Clearly, if The Distinct Spheres Condition fails to hold for a connected infinite graph $\Gamma$ and if there exist a pair of vertices $\gamma, \delta\in V\Gamma$ and $n\in\N$ such that $S(\gamma, n) = S(\delta, n)$, then $S(\gamma, m) = S(\delta, m)$ for all $m \geq n$.  Observe that if $\Gamma$ satisfies the DSC, then $\Gamma$ has infinite diameter.

Other sufficient conditions for 2-distinuishability can be found in the literature. Indeed, the so called {\em weak-e.c.}\,condition, defined by Bonato and Deli\'{c} in \cite{BoDe}, is a sufficient condition for a denumerable  graph to be $2$-distinguishable; unlike graphs satisfying the Distinct Spheres Condition, graphs satisfying the weak-e.c.\,condition have diameter $2$ and no vertex of finite valence.

\\

An outline of this paper is as follows.  In Section 2, we show that the Distinct Spheres Condition  implies $2$-distinguishability for connected denumerable graphs of infinite diameter.  In Section 3, we show that connected denumerable primitive graphs satisfy the DSC.  In Section 4, we show that denumerable vertex-transitive graphs of connectivity $1$ and the Cartesian product of connected denumerable graphs of infinite diameter satisfy the DSC. In Section 5, we consider the Infinite Motion Conjecture \cite{tucker:distinguishing_maps}, that if a locally finite graph has infinite motion (i.e. every non-identity automorphism has infinite support), then the graph has distinguishing number $2$; we show, in fact, that the DSC implies infinite motion.  In Section 6, we present an infinite class of denumerable, vertex-transitive graphs not satisfying the DSC that are $2$-distinguishable and have infinite motion, showing that the DSC is in no way necessary.

%
%
\section{A rather useful lemma}

In this paper our graphs have no loops or multiple edges.  All our graphs are connected. The distance function in a graph $\Gamma$ is denoted by $d_\Gamma$, but more simply by $d$ when $\Gamma$ is the only graph under consideration.  For a vertex $\alpha \in V\Gamma$, we define the {\em sphere of radius $r$ with center $\alpha$} to be
\[S(\alpha, r) = \{\beta \in V\Gamma :  d(\alpha, \beta) = r \},\]
and the {\em ball of radius $r$ with center $\alpha$} to be
\[B(\alpha, r) = \{\beta \in V\Gamma :  d(\alpha, \beta) \leq r \}.\] 
The {\em diameter} of $\Gamma$ is $\sup \{d_\Gamma(\alpha, \beta) : \alpha, \beta \in V\Gamma\}$, which may, of course, be infinite.

If $G\leq\aut(\Gamma)$ and $Y \subseteq V\Gamma$, we define the {\em setwise stabilizer of $Y$ in $G$} to be the subgroup
\[G_{\{Y\}} = \{g \in G :  y^g \in Y \text{ for all } y \in Y\},\]
and the {\em pointwise stabilizer of $Y$ in $G$} to be the subgroup
\[G_{(Y)} = \{g \in G :  y^g =y \text{ for all } y \in Y\}.\]

The set of positive integers is denoted by $\N$.

\\

The following useful lemma is sometimes all that is needed to determine whether a given infinite graph has distinguishing number $2$.

\begin{distinct_spheres_lemma} \label{lemma:distinguishing} Suppose that $\Gamma$ is a connected denumerable graph. If $\Gamma$ satisfies The Distinct Spheres Condition, then $\aut(\Gamma)$ (and therefore  every subgroup of $\aut(\Gamma)$ as well as $\Gamma$ itself) has distinguishing number $2$.
\end{distinct_spheres_lemma}

\begin{proof}
Let $\alpha \in V\Gamma$ be as described in the DSC, and write $A = \aut(\Gamma)$. 

We enumerate all 2-subsets $\{\gamma, \delta\}$ of vertices equidistant from $\alpha$ as follows. Let $P = \big\{\{\gamma, \delta\} :  \gamma, \delta \in V \text{ distinct and } d(\alpha, \gamma)= d(\alpha, \delta)\big\}$. Because $V\Gamma$ is denumerable, so is $P$.
 We may thus fix an enumeration of $P$ and write $P = \{\{\gamma_i, \delta_i\} :   i \in\N\}$, with $\gamma_1, \delta_1 \in S(\alpha, 1)$.

Consider initially the pair $(\gamma_1, \delta_1)$, and arbitrarily choose $N_1\geq 3$. By the Distinct Spheres Condition, there exists an integer $n_1 > N_1$ such that the symmetric difference $S(\gamma_1, n_1) \symdiff S(\delta_1, n_1)$ is not empty. Choose a vertex $\beta_1 \in S(\gamma_1, n_1) \symdiff S(\delta_1, n_1)$. If $\gamma_1^g = \delta_1$ for some $g\in A$, then $\left ( S(\gamma_1, n_1) \right )^g =S(\delta_1, n_1)$, which implies $\beta_1^g \not = \beta_1$. Thus $\delta_1 \not \in \gamma_1^{A_{\beta_1}}$ and, symmetrically, $\gamma_1\notin\delta_1^{A_{\beta_1}}$.

For $i > 1$ we proceed to choose $N_i, n_i \in \mathbb{N}$, and $\beta_i \in V\Gamma$ inductively. Suppose we have already chosen suitable $N_{i-1}, n_{i-1} \in \N$ and $\beta_{i-1} \in V\Gamma$. Let $N_i = d(\alpha, \gamma_i) + d(\alpha, \beta_{i-1}) + 1$, and so $N_i = d(\alpha, \delta_i) + d(\alpha, \beta_{i-1}) +1$. There exists a least integer $n_i > N_i$ such that $S(\gamma_i, n_i) \symdiff S(\delta_i, n_i)$ is not empty, and we may thus choose $\beta_i \in S(\gamma_i, n_i) \symdiff S(\delta_i, n_i)$. (Observe that, for each $i\in\N$, the vertex $\beta_{i-1}$ determines the integer $N_i$, which determines $n_i$, which in turn determines $\beta_i$, and so on.)  Again we note that if $\gamma_i^g = \delta_i$ for some $g \in A$, then $\left ( S(\gamma_i, n_i) \right )^g = S(\delta_i, n_i)$, and so $\beta_i^g \not = \beta_i$. Thus $\delta_i \not \in \gamma_i^{A_{\beta_i}}$ and, symmetrically, $\gamma_i\notin\delta_i^{A_{\beta_i}}$.  

Let $Y = \{\beta_i :   i=1,2,\ldots\}$.   (As each vertex $\beta_i$ is determined by its predecessor $\beta_{i-1}$, the Axiom of Choice does not come into play in defining the set $Y$.)

We claim that 
\begin{equation}\label{distancing_betas}
d(\alpha, \beta_i) > d(\alpha, \beta_{i-1})+1\quad{\text{for all }}i\geq2.
\end{equation}
Indeed, observe that for each $i\in\N$, we have for some $\zeta_i\in\{\gamma_i,\delta_i\}$,
\begin{eqnarray*}
d(\alpha,\beta_{i-1})&=&N_i-d(\alpha,\zeta_i)-1\\
&\leq& n_i-d(\alpha,\zeta_i)-2\\
&=&d(\beta_i,\zeta_i)-d(\alpha,\zeta_i)-2\\
&\leq& d(\alpha,\beta_i)-2.
\end{eqnarray*}
By iteration of inequality (\ref{distancing_betas}),  we have
\begin{equation}\label{iterate}
 d(\alpha,\beta_i)> d(\alpha,\beta_j)+i-j\quad{\text{for}}\ 1\leq j<i,
\end{equation} 
and, in particular, 
\begin{equation}\label{distancefrom_beta1}
d(\alpha, \beta_i) > d(\alpha, \beta_1)+i-1>2.
\end{equation}

By the inequality (\ref{iterate}), no two elements of $Y$ are equidistant from $\alpha$, and no two elements of $Y$ are adjacent to each other.
For if $\beta_i$ and $\beta_j$ were adjacent, then $|d(\alpha,\beta_i)-d(\alpha,\beta_j)|\leq1$ would have to hold, contrary to (\ref{iterate}).  It follows from inequality (\ref{distancefrom_beta1}) that $Y \cap B(\alpha, 1) = \emptyset$.

Hence the subgroup $A_{\alpha, \{Y\}}$ fixes the set $Y$ pointwise. Thus for each $i \geq 1$ we have that $A_{\alpha, \{Y\}} \leq A_{\alpha, \beta_i}$, and therefore
\begin{equation*}
\delta_i \not \in \gamma_i^{A_{\alpha, \{Y\}}}\quad{\text{and}}\quad\gamma_i\notin\delta_i^{A_{\alpha,\{Y\}}}.
\end{equation*}

Now $\delta \in \gamma^{A_\alpha}$ only if $\delta$ and $\gamma$ are equidistant from $\alpha$.  But the set $P$ contains {\em all} pairs of distinct vertices equidistant from $\alpha$. Thus $A_{\alpha, \{Y\}}$ fixes every vertex of $\Gamma$, and because the action of $A$ on $V\Gamma$ is faithful, $A_{\alpha, \{Y\}} = \langle 1 \rangle$.

Let $Y' = Y \cup B(\alpha, 1)$. There is no element $\zeta \in Y'\setminus \{\alpha\}$ such that $B(\zeta, 1) \subseteq Y'$, and so  $A_{\{Y'\}} \leq A_\alpha$.  
Since every automorphism of $\Gamma$ that fixes $Y'$ setwise must fix $\alpha$, and $Y \cap B(\alpha, 1) = \emptyset$, we have $A_{\{Y'\}} \leq A_{\alpha, \{Y\}} = \langle 1 \rangle$,  which means that $A$ has distinguishing number $2$.
\end{proof}

\begin{rem} 
{\normalfont
Observe that the set $Y$ constructed in the above proof has the property that $Y \cap B(\alpha,N_1-1)=\emptyset$.  This holds because $N_1\leq n_1-1\leq d(\alpha,\beta_1)<d(\alpha,\beta_i)$ for all $i>1$.  Since $N_1\geq3$ may be arbitrarily large, aside from the vertex $\alpha$ and its neighborhood, the  2-distinguishing set $Y'$ can be chosen to be arbitrarily far removed from the vertex $\alpha$.
}
\end{rem}

%
%
\section{Distinguishability of infinite primitive groups and graphs}

\begin{lemma} \label{lemma:prim_sphere_difference} If $\Gamma$ is a primitive graph with infinite diameter, then for all distinct $\gamma, \delta \in V\Gamma$ and all integers $n\geq 1$,
\[S(\gamma, n) \not = S(\delta, n).\]
In particular, $\Gamma$ satisfies the Distinct Spheres Condition.
\end{lemma}

\begin{proof} Suppose $\Gamma$ is a primitive graph with infinite diameter, and write $A = \aut(\Gamma)$.  Fix $n \geq 1$, and define the equivalence relation $\sim$ on $V\Gamma$ by $\gamma \sim \delta$ if and only if $S(\gamma, n) = S(\delta, n)$.  Since graph automorphisms preserve distance, the relation $\sim$ is an $A$-invariant relation on $V\Gamma$.

Since $A$ is primitive, the only $A$-invariant equivalence relations on $V\Gamma$ are the trivial relation and universal relation.  First suppose that $\sim$ is universal, that is, $S(\gamma, n) = S(\delta, n)$ for all $\gamma, \delta \in V\Gamma$.  Since $\Gamma$ is connected and has infinite diameter, one may choose $\gamma$ and $\delta$ such that $d(\gamma,\delta)=n$.  But then $\delta\in S(\gamma,n)\setminus S(\delta,n)$, giving a contradiction.  It follows that $\sim$ is trivial, as required. \end{proof}

We are now able to prove straightforwardly the first result stated in our introduction.

\begin{citethm}{\ref*{thm:primitive_distinguishing_number_2}}  Every connected denumerable primitive graph with infinite diameter has distinguishing number $2$.
\end{citethm}

\begin{proof} Suppose $\Gamma$ is denumerable and primitive with infinite diameter.  Because $\Gamma$ is primitive, it is connected.  By Lemma~\ref{lemma:prim_sphere_difference}, $\Gamma$ satisfies the Distinct Spheres Condition.
Thus $\Gamma$ has distinguishing number $2$ by The Distinct Spheres Lemma.
\end{proof}

Theorem~\ref{thm:primitive_distinguishing_number_2} fails if one removes the requirement that $\Gamma$ have infinite diameter. As an extreme example, the complete denumerable graph, though primitive, has distinguishing number $\aleph_0$.

\begin{citecor}{\ref*{cor:distinguish_inf_diam_prim_groups}}  If $V$ is a denumerable set and $G \leq \sym(V)$ is primitive and has a connected
orbital graph of infinite diameter, then the distinguishing number of $G$ is
$2$. \qed
\end{citecor}

\begin{citecor}{\ref*{cor:distinguish_subdegree_finite_prim_groups}} If $G\leq \sym(V)$ is infinite and primitive and all suborbits of $G$ are finite, then $G$ has distinguishing number $2$.
\end{citecor}

\begin{proof} Since $G$ is primitive, for any distinct elements $\alpha, \beta \in V$ the graph $\Gamma = (V, \{\alpha, \beta\}^G)$ is connected, infinite, edge-transitive, and primitive.  
That all suborbits are finite means that there are only finitely many vertices $\gamma$ such that $\{\alpha,\gamma\}\in\{\alpha, \beta\}^G$.  Since $\Gamma$ is vertex-transitive, it must be locally finite.  By Corollary~\ref{primgraph_dist=2}, the graph $\Gamma$ and hence the group $G$ have distinguishing number $2$.
\end{proof}
\\

%
%
\section{Further applications of the Distinct Spheres Lemma}

In this section we apply the Distinct Spheres Lemma to prove the $2$-distinguishability of some other types of infinite graphs. 

\\

A connected graph $\Gamma$ has {\em connectivity} 1 if there exists a vertex $\alpha \in V\Gamma$ such that $\Gamma \setminus \{\alpha\}$ is not connected;
such a vertex $\alpha$ is called a {\em cut vertex}. 

\begin{thm} \label{thm:connectivity_one} Suppose $\Gamma$ is a denumerable connected graph such that for every vertex $\delta \in V\Gamma$ the graph $\Gamma \setminus \{\delta\}$ has at least two infinite components. Then $\Gamma$ is 2-distinguishable.
\end{thm}

\begin{proof} Fix $\alpha \in V\Gamma$ and choose a pair of vertices $\gamma, \delta \in V\Gamma$ equidistant from $\alpha$. (Since every vertex is a cut vertex, any vertex would be suitable a suitable choice for $\alpha$.)  By assumption, the graph $\Gamma \setminus \{\delta\}$ contains at least two infinite components; let $C$ be such a component with $\gamma \notin VC$.  It is immediate by induction that $C$ has infinite diameter.  Hence, for all $n \in \mathbb{N}$ we have that $S(\delta, n) \cap VC \neq \emptyset$.  Since any path in $\Gamma$ joining $\gamma$ to some vertex $\beta \in S(\delta, n) \cap VC$ must contain $\delta$, we have $d(\gamma, \beta) = d(\gamma, \delta) + d(\delta, \beta) > n$. Hence $S(\gamma, n) \not = S(\delta, n)$, and so $\Gamma$ has distinguishing number $2$ by The Distinct Spheres Lemma.
\end{proof}  

\begin{cor}  \cite{watkins:distinguishability} Every denumerable tree without vertices of valence $1$ is $2$-distinguishable.
\end{cor}

The following theorem, stated in our introduction, is an immediate consequence of Theorem~\ref{thm:connectivity_one}.

\begin{citethm}{\ref*{thm:vertex-trans}} Every denumerable vertex-transitive graph of connectivity 1 has distinguishing number $2$.
\end{citethm}

The {\em Cartesian product} $\Gamma=\Lambda \cart \Theta$ of graphs $\Lambda$ and $\Theta$ is the graph with vertex set $V\Gamma=V\Lambda \times V\Theta$, where vertices $(\lambda, \theta), (\lambda', \theta')$ are adjacent in $\Gamma$ if and only if either $\lambda = \lambda'$ and $d_{\Theta}(\theta, \theta') = 1$, or $\theta = \theta'$ and $d_\Lambda(\lambda, \lambda') = 1$. In particular, distance in $\Gamma$ satisfies:
\[d_\Gamma((\lambda, \theta), (\lambda', \theta')) = d_\Lambda(\lambda, \lambda') + d_\Theta(\theta, \theta').\]  
The graphs $\Lambda$ and $\Theta$ are called {\em factors} of $\Gamma$.  It is easy to see that the operation of Cartesian product is associative and commutative. A Cartesian product of two graphs is connected if and only if all of its factors are connected.
Our next result establishes the distinguishability of the Cartesian product of two graphs, but it can of course be extended inductively to the Cartesian product of finitely many graphs.

\begin{citethm}{\ref*{thm_cart_prod}} The Cartesian product of any two connected denumerable graphs of infinite diameter has distinguishing number $2$.
\end{citethm}

\begin{proof}  Let $\Lambda$ and $\Theta$ be connected denumerable graphs having infinite diameter, and let $\Gamma=\Lambda\square\Theta$.  (Spheres within any of these three graphs are understood to be defined in terms of the distance metric for that graph.)   For all $n\in\N$ and all $\lambda, \lambda' \in V\Lambda$ and $\theta, \theta' \in V\Theta$, we will show that $(\lambda,\theta)\neq (\lambda',\theta')$ implies $S((\lambda,\theta),n)\neq S((\lambda', \theta'),n)$, from which the theorem follows by the Distinct Spheres Lemma.

Suppose that $S((\lambda,\theta),n)=S((\lambda', \theta'),n)$ for some $n$ (and hence for all integers greater than $n$).  Define $j = d_\Lambda(\lambda, \lambda')$ and $k = d_\Theta(\theta, \theta')$.   

Since $\Lambda$ is connected and has infinite diameter, for any vertex $\mu\in S(\lambda,n)$ we have nonvacuously that $d_\Gamma((\lambda,\theta),(\mu,\theta'))=d_\Lambda(\lambda, \mu)+d_\Theta(\theta,\theta')=n+k$.  Hence $(\mu,\theta')\in S((\lambda,\theta),n+k)=S((\lambda',\theta'),n+k)$ by our assumption.  This yields $d_\Lambda(\lambda',\mu)=d_\Gamma((\lambda',\theta'),(\mu,\theta'))=n+k.$  Thus $\mu\in S(\lambda',n+k)$.  Next we swap $\lambda$ with $\lambda'$ and replace $n$ by $n+k$ in the foregoing argument to obtain
\begin{equation*}
S(\lambda,n)\subseteq S(\lambda', n+k)\subseteq S(\lambda, n+2k).
\end{equation*}
Hence $k=0$ and therefore $\theta = \theta'$.

Similarly, since $\Theta$ is connected and has infinite diameter, $d_\Gamma((\lambda, \theta), (\lambda', \theta'))$ $= d_\Theta(\theta, \theta')+j$, and we deduce that $S(\theta,n)\subseteq S(\theta, n+2j)$.  Hence $j = 0$ and therefore $\lambda = \lambda'$.
\end{proof}

%
%

\section{The infinite motion conjecture} \label{section:infinite_motion}
Given a permutation group $G\leq \sym(V)$ and a permutation $g\in G$, the {\em support} of $g$ is the set $\{\alpha\in V:\alpha^g\neq\alpha\}$ and the {\em motion} of $g$, denoted $m(g)$, is the cardinality of its support. The {\em motion} of $G$, denoted $m(G)$, is defined to be  $\inf\{m(g) : g \in G\}$.   If $G$ contains no element of finite support, then $G$ is said to have {\em infinite motion}. 

The motion of $G$ is usually called the {\em minimal degree} of $G$. However, in the first paper \cite{RS} to exploit effectively the minimal degree of a group in the context of distinguishability of graphs, the term ``motion" is used, as is done here.

The Motion Lemma of Russell and Sundaram \cite{RS} states that if $G$ and $V$ are finite and $m(G)>2\log_2(|G|)$, then $G$ has distinguishing number 2.  Their proof is a short and elegant application of the probabilistic method.  (See \cite{CT} for applications of the Motion Lemma in a wide variety of contexts.)  Thus, in the finite case, when $m(G)$ is large enough compared to $|V|$, then $G$ has distinguishing number 2. When $G$ is infinite, it appears that ``large enough" just means infinite.  For the automorphism group of a graph, we have the following conjecture (see \cite{tucker:distinguishing_maps}).

\begin{InfiniteMotionConjecture}  If $\Gamma$ is a connected, locally finite, denumerable  graph whose automorphism group has infinite motion, then $\Gamma$ has distinguishing number 2.
\end{InfiniteMotionConjecture}

There is a direct connection between the DSC and infinite motion.

\begin{thm} \label{DSL_implie_infinite_motion} If a connected graph $\Gamma$ satisfies the Distinct Spheres Condition, then $\aut(\Gamma)$ has infinite motion.
\end{thm}

\begin{proof}  Suppose $\Gamma$ satisfies the DSC, and let $g\in\aut(\Gamma)$. If $g$ fixes no vertex in $\Gamma$, then $m(g)$ is infinite.  Suppose rather that $g$ fixes some $\gamma \in V\Gamma$.  If $\alpha\in V\Gamma$ and $\beta=\alpha^g$ (which means that $\alpha$ and $\beta$ are equidistant from $\gamma$), then for all $n\in\N$, the automorphism $g$ fixes no vertex in the set $S(\alpha, n) \Delta S(\beta, n)$.  This set is nonempty for infinitely many $n\in\N$ by the DSC.   Hence $g$ has infinite support, and again $m(g)$ is infinite.
\end{proof}

A consequence of Theorem \ref{DSL_implie_infinite_motion} is that all the graphs considered in this paper, when restricted to the infinite but locally finite case, have infinite motion.  The list of such graphs includes primitive graphs, vertex-transitive graphs of connectivity $1$, and Cartesian products. Thus if the Infinite Motion Conjecture is true, the 2-distinguishability of these infinite, locally finite graphs would follow directly.

%
%

\section{A rather unnecessary example}

We have shown that every connected denumerable graph satisfying the Distinct Spheres Condition is $2$-distinguishable.  We present a family of examples demonstrating that, even for vertex-transitive graphs, this condition is not necessary. These examples also have infinite motion.  Thus, infinite motion does not imply the DSC.

Let $k$ be any odd integer such that $k\geq7$.  The $k$-cycle $C_k$ is obviously connected and vertex-transitive.  Since $k\geq7$, it is 2-distinguishable \cite{AC1}.  Since $k$ is odd, for any integer $n\geq k$ and all $\zeta,\eta\in VC_k$, there exists a walk of length $n$ joining $\zeta$ and $\eta$.  (Recall that in a walk, as opposed to a path, backtracking is permitted.).  Let $T$ denote the $r$-valent tree where $r\geq2$.   We know $T$  to be vertex-transitive, 2-distinguishable, and of infinite diameter.

Our desired example $\Gamma$ is the {\em direct product} (also known as the {\em weak product} or {\em categorical product}) $T\otimes C_k$. Thus $V\Gamma=VT\times VC_k$ and $\{(\alpha,\zeta),(\beta,\eta)\}$ is an edge of $\Gamma$ if and only if both $\{\alpha,\beta\}$ is an edge of $T$ and $\{\zeta,\eta\}$ is an edge of $C_k$.  Since $k$ is odd, $\Gamma$ is connected (by Theorem 1 in  \cite{Weichsel}). By \cite[pp.\,177--178]{IK}, if $g\in\aut{\Gamma}$, then there exist $g_1\in\aut(T)$ and $g_2\in\aut(C_k)$ such that 
\begin{equation*}
(\alpha,\zeta)^g=(\alpha^{g_1},\zeta^{g_2})\quad{\text{for all}}\ (\alpha,\zeta)\in V\Gamma.
\end{equation*}  
Since $T$ and $C_k$ are vertex-transitive, so is $\Gamma$.  Furthermore, every automorphism of $\Gamma$ induces a permutation of the set
\begin{equation}\label{level}
\{H_\alpha :   \alpha \in VT\},
\end{equation}
where $H_\alpha = \{(\alpha,\zeta) :   \zeta\in VC_k\}$.

We claim that the graph $\Gamma$ fails to satisfy the Distinct Spheres Condition.  Let $\{\alpha,\beta\}$ be an edge of $T$, and let vertices $\eta_1$ and $\eta_2$ be neighbors of $\zeta$ in $C_k$.  Then in $\Gamma$, the vertices $(\beta,\eta_1)$ and $(\beta,\eta_2)$ belong to $S\big((\alpha,\zeta),1\big)$.  We leave it to the reader to verify that for all $n\geq k$, the $n$-spheres of $(\beta,\eta_1)$ and $(\beta,\eta_2)$ are identical; specifically, they are of the form $\{(\gamma,\theta) :   \theta\in VC_k; d_T(\beta,\gamma)=n\}$.

It remains only to show that, despite the failure by $\Gamma$ just described, $\Gamma$ is nonetheless 2-distinguishable.  By our assumptions, $T$ admits a distinguishing 2-coloring $f:VT\rightarrow\{c_1,c_2\}$, while $C_k$ admits a distinguishing 2-coloring $h:VC_k\rightarrow\{c_1,c_2\}$.  Since $|VC_k|$ is odd, no automorphism of $C_k$ interchanges $h^{-1}(c_1)$ with $h^{-1}(c_2)$.  The following function $F:V\Gamma\rightarrow\{c_1,c_2\}$ is a distinguishing 2-coloring of $\Gamma$:
\begin{eqnarray*}
F(\alpha,\zeta)&=&c_1\quad{\text{if}}\quad f(\alpha)=h(\zeta);\\
F(\alpha,\zeta)&=&c_2\quad{\text{if}}\quad f(\alpha)\neq h(\zeta).  
\end{eqnarray*}
One easily sees that the 2-coloring determined by $F$ is preserved by no automorphism that permutes the sets in Equation (\ref{level}).   By the above remarks these are the only automorphisms of $\Gamma$.

Finally, we show that $\Gamma$ has infinite motion.  Since the tree $T$ has infinite motion, we only need to show that any automorphism of $\Gamma$ that fixes setwise each set in (\ref{level}) moves infinitely many vertices.  But any such automorphism $g$ moving a vertex of the form $(\alpha, \zeta)$ must also move some vertex of the form $(\beta, \eta)$ for each $\beta$ adjacent to $\alpha$ in $T$. Continuing to the neighbors of $\beta$ and onward to its neighbors, we see that  for each vertex $\beta\in VT$, the automorphism $g$ must move at least one vertex in $H_\beta$.  Thus the automorphism has infinite motion.

%
%

\end{document}